\documentclass[12pt]{amsart}

\usepackage[bookmarks=true]{hyperref}
\hypersetup{colorlinks,citecolor=blue,linkcolor=blue}

\newtheorem{theorem}{Theorem}
\newtheorem{lemma}{Lemma}[section]

\newtheorem{question}{Question}
\theoremstyle{definition}

\newtheorem{rmk}[lemma]{Remark}

\newcommand\Hy{\mathbb{H}}

\newcommand\F{\mathbb{F}}

\newcommand\Z{\mathbb{Z}}
\newcommand\Q{\mathbb{Q}}

\newcommand\C{\mathbb{C}}
\newcommand\cO{\mathcal{O}}

\newcommand\sys{{\rm sys}}
\newcommand\sysg{{\rm sysg}}

\newcommand\vol{{\rm vol}}
\newcommand\Nfr{\mathcal{N}_{fr}}
\newcommand\Nr{{\rm Norm}}
\newcommand\sss{\mathbb{S}^1}

\newcommand{\pinj}{$\pi_1$-injective }

\DeclareMathOperator{\SL}{SL}
\DeclareMathOperator{\GL}{GL}

\DeclareMathOperator{\PSL}{PSL}

\begin{document}
\title{Free subgroups of $3$-manifold groups}


\author{Mikhail Belolipetsky}\thanks{Belolipetsky is partially supported by CNPq, FAPERJ and MPIM in Bonn.}
\author{Cayo D\'oria}\thanks{D\'oria was partially supported by CNPq and FAPESP}
\address{
IMPA\\
Estrada Dona Castorina, 110\\
22460-320 Rio de Janeiro, Brazil}
\email{mbel@impa.br}
\address{
Departamento de Matem\'atica Aplicada, IME-USP\\ Rua do Mat\~ao, 1010, Cidade Universit\'aria\\
05508-090, S\~ao Paulo SP, Brazil.}
\email{cayofelizardo@ime.usp.br}

\begin{abstract}
We show that any closed hyperbolic $3$-manifold $M$ has a co-final tower of covers $M_i \to M$ of degrees $n_i$ such that any subgroup of $\pi_1(M_i)$ generated by $k_i$ elements is free, where $k_i \ge n_i^C$ and $C = C(M) > 0$. Together with this result we prove that $\log k_i \ge C_1 \sys_1(M_i)$, where $\sys_1(M_i)$ denotes the systole of $M_i$, thus providing a large set of new examples for a conjecture of Gromov. In the second theorem $C_1> 0$ is an absolute constant. We also consider a generalization of these results to non-compact finite volume hyperbolic $3$-manifolds.  
\end{abstract}

\maketitle

\section{Introduction}

Let $\Gamma < \PSL_2(\C)$ be a cocompact Kleinian group and $M = \Hy^3/\Gamma$ be the associated quotient space. It is a closed orientable hyperbolic $3$-orbifold, it is a manifold if $\Gamma$ is torsion-free. We will call a group $\Gamma$ \emph{$k$-free} if any subgroup of $\Gamma$ generated by $k$ elements is free. We denote the maximal $k$ for which $\Gamma$ is $k$--free by $\Nfr(\Gamma)$ and we call it the \emph{free rank} of $\Gamma$. For example, if $S_g$ is a closed  Riemann surface of genus $g$, then its fundamental group satisfies $\Nfr(\pi_1(S_g)) = 2g-1$.  In this note we prove that for any Kleinian group as above there exists an exhaustive filtration of normal subgroups $\Gamma_i$ of $\Gamma$ such that $\Nfr(\Gamma_i) \ge [\Gamma:\Gamma_i]^C$, where $C = C(\Gamma) > 0$ is a constant. In geometric terms the result can be stated as follows.

\begin{theorem}\label{thm1}
Let $M$ be a closed hyperbolic $3$-orbifold. Then there exists a co-final tower of regular finite-sheeted covers $M_i \to M$ such that 
$$\Nfr(\pi_1(M_i)) \ge \vol(M_i)^C,$$ 
where $C = C(M)$ is a positive constant which depends only on $M$.
\end{theorem}

The proof of the theorem is based on the previous results of Baumslag, Shalen and Wagreich \cite{BSh, SW},  Belolipetsky \cite{Bel}, and Calegari--Emerton \cite{CE}. Let us emphasize that although some of the results use arithmetic techniques, our theorem applies to \emph{all} closed hyperbolic $3$-orbifolds. A result of similar flavor but for another property of $3$-manifold groups was obtained by Long, Lubotzky and Reid in \cite{LLR}. Indeed, in some parts our construction comes close to their argument. 

Together with Theorem~\ref{thm1} we obtain the following theorem of independent interest:

\begin{theorem}\label{thm2} Any closed hyperbolic $3$-orbifold admits a sequence of regular manifold covers $M_i \to M$ such that 
$$\Nfr(\pi_1(M_i)) \ge (1 + \varepsilon)^{\sys_1(M_i)}, $$
where $\varepsilon > 0$ is an absolute constant and $\sys_1(M_i)$ is the length of a shortest closed geodesic in $M_i$. 
\end{theorem}

This type of bound was stated by Gromov \cite[Section~5.3.A]{Gromov:hyp groups} for hyperbolic groups in general, but later turned into a conjecture (see \cite[Section~2.4]{Gromov:GAFA}). We refer to the introduction of \cite{Bel} for a related discussion and some other references. In \cite{Gromov:GAFA}, Gromov particularly mentioned that the conjecture is open even for hyperbolic $3$-manifold groups. 
The first set of examples of hyperbolic $3$-manifolds for which the conjecture is true was presented in \cite{Bel}. These examples were all arithmetic. Our theorem significantly enlarges this set.  

We review the construction of covers $M_i \to M$ and prove a lower bound for their systoles in Section~\ref{sec:prelim}. Theorems~\ref{thm1} and \ref{thm2} are proved in Section~\ref{sec3}. 
In Section~\ref{sec4} we consider a generalization of the results to non-compact finite volume $3$-manifolds. Their groups always contain a copy of $\Z\times\Z$, so have  $\Nfr = 1$, however, we can modify the definition of the free rank so that it becomes non-trivial for the non-compact manifolds: we define $\Nfr'(\Gamma)$ to be the maximal $k$ for which the group $\Gamma$ is $k$--semifree, where $\Gamma$ is called \emph{$k$-semifree} if any subgroup generated by $k$ elements is a free product of free abelian groups. With this definition at hand we can extend Gromov's conjecture to the groups of finite volume non-compact manifolds. In Section~\ref{sec4} we prove:  

\begin{theorem}\label{thm3} Any finite volume hyperbolic $3$-orbifold admits a sequence of regular manifold covers $M_i \to M$ such that 
$$\Nfr'(\pi_1(M_i)) \ge (1 + \varepsilon)^{\sys_1(M_i)}, $$
where $\varepsilon > 0$ is an absolute constant. 
\end{theorem}

To conclude the introduction we would like to point out one important detail. While in Theorems~\ref{thm2} and \ref{thm3} we have an absolute constant $\varepsilon > 0$, the constant in Theorem~\ref{thm1} depends on the base manifold. In \cite{Bel} it was shown that in arithmetic case $C(M)$ is also bounded below by a universal positive constant. Existence of a bound of this type in general remains an open problem. 

\begin{question}\label{quest}
Do there exist an absolute constant  $C_0 > 0$ such that for any $M$ in Theorem~\ref{thm1} we have
$C(M) \ge C_0.$
\end{question}

\section{Preliminaries}\label{sec:prelim}

Let $\Gamma < \PSL_2(\C)$ be a \emph{lattice}, i.e. a finite covolume discrete subgroup. By Mostow--Prasad rigidity, $\Gamma$ admits a discrete faithful representation into $\SL_2(\C)$ with the entries in some (minimal) number field $E$. Since $\Gamma$ is finitely generated, there is a finite set of primes $S$ in $E$ such that $\Gamma < \SL_2(\cO_{E,S})$, where $\cO_{E,S}$ denotes the ring of $S$-integers in $E$. 

Following  Calegari--Emerton \cite{CE}, we can consider an exhaustive filtration of normal subgroups $\Gamma_i$ of $\Gamma$ which gives rise to a co-final tower of hyperbolic $3$-manifolds covering $\Hy^3/\Gamma$. 
The subgroups $\Gamma_i$ are defined as follows. From the description of $\Gamma$ given above it follows that it is residually finite and for all but finitely many primes $\mathfrak{p} \in \cO_E$ there is an injective map $\phi_{\mathfrak{p}} : \Gamma \to \SL_2(\hat{\cO}_{E, \mathfrak{p}})$ (where $\hat{\cO}_{E, \mathfrak{p}}$ denotes the $\mathfrak{p}$-adic completion of the ring of integers of $E$). Let $p$ be a rational prime such that for any prime $\mathfrak{p}$ in $\cO_E$ which divides $p$, the correspondent map $\phi_{\mathfrak{p}}$ is injective (this holds for almost all primes $p$). We can write $p\cO_E= \mathfrak{p}_1^{e_1} \cdots \mathfrak{p_m}^{e_m}$.

For any $j=1, \ldots , m$, the ring $\hat{\cO}_{E,\mathfrak{p}_j}$ contains $\Z_p$ as a subring and is a $\Z_p$--module of dimension $d_j=e_jf_j$, where $f_j$ is the degree of the extension of residual fields $[\cO_E / \mathfrak{p}_j : \Z / p\Z]$. 
If we fix $j$ and a basis $b_1^j, \ldots, b_{d_j}^j \in \hat{\cO}_{E,\mathfrak{p}_j}$ as $\Z_p$--module, we have a natural ring homomorphism $\psi_j: \hat{\cO}_{E,\mathfrak{p}_j} \rightarrow M_{d_j \times d_j}(\Z_p)$ given by $\psi_j(x)=(x_{rs})$ if $xb_s^j= \sum_{r=1}^{d_j} x_{rs} b_r^j$.

Let $\psi: \prod_{j=1}^m \SL_2(\hat{\cO}_{E,\mathfrak{p}_j}) \rightarrow \GL_N(\Z_p)$ be given diagonally by the blocks $\psi_1, \ldots, \psi_m$, where $N=2 \sum_j d_j$.                                                     
Let $\phi = \psi \circ \prod_{j=1}^m \phi_{\mathfrak{p}_j} : \SL_2(\cO_{E,S}) \rightarrow \GL_N(\Z_p)$. The Zariski closure of the image of $\phi$ is a group $G < \GL_N(\Z_p)$ of dimension $d \ge 6$ (cf. \cite[Example~5.7]{CE}). It is a $p$-adic analytic group which admits a
normal exhaustive filtration
$$ G_i = G \cap \mathrm{ker}\big(\GL_N(\Z_p) \to \GL_N(\Z_p /p^i\Z_p)\big).$$
This filtration gives rise to a filtration of $\Gamma$ via the normal subgroups $\Gamma_i = \phi^{-1}(G_i)$. The filtration $(\Gamma_i)$ is exhaustive because $\phi$ is injective.

Associated to each of the subgroups $\Gamma_i$ of $\Gamma$ is a finite-sheeted cover $M_i$ of $M = \Hy^3/\Gamma$, and by the construction the sequence $(M_i)$ is a co-final tower of covers of $M$. By Minkowski's lemma, almost all groups $G_i$ are torsion-free, hence associated $M_i$ are smooth hyperbolic $3$-manifolds. Therefore, when it is needed we can assume that $M$ is a manifold itself. 

We will require a lower bound for the systole of $M_i$. Such a bound is essentially provided by Proposition~10 of \cite{GL14}, which can be seen as a generalization of a result of Margulis \cite{M} (see also \cite{LLR}). The main difference is that we do not restrict to arithmetic manifolds. The main technical difference is that while in [op. cit.] the authors consider matrices with real entries we do it for $p$-adic numbers, which requires replacing norm of a matrix by the height of a matrix. This technical part is more intricate, however, as it is shown below, it does not affect the main argument. 

\begin{lemma} \label{lem} Suppose $M$ is a compact manifold. Then there is a constant $c_1 = c_1(M) > 0$ such that $\sys_1(M_i) \ge c_1 \log n_i$, where $n_i = [\Gamma:\Gamma_i]$.
\end{lemma} 

\begin{proof}
Since $M$ is compact, we can apply the Milnor--Schwarz lemma. Therefore, if we fix a point $o \in \Hy^3$, then $\Gamma$ has a finite symmetric set of generators $X$ such that the map $(\Gamma,X) \rightarrow \Hy^3$ given by $\gamma \mapsto \gamma(o)$ is a $(C_1,C_2)$ quasi-isometry.
This means that for any pair $\gamma_1, \gamma_2 \in \Gamma$ we have $$ C_1 d_X(\gamma_1,\gamma_2) - C_2 \leq d(\gamma_1(o),\gamma_2(o)) \leq \frac{1}{C_1} d_X(\gamma_1,\gamma_2) + C_2,$$
where $d(\cdot,\cdot)$ denotes the distance function in $\Hy^3$, $d_X(\gamma_1,\gamma_2)=|\gamma_1^{-1}\gamma_2|_X$ and $|\gamma|_X$ is the minimal length of a word in $X$ which represents $\gamma$. 
For any $i \geq 1$, we define $\sys(\Gamma_i,X) = \min \{ d_X(1,\gamma) \, | \, \gamma \in \Gamma_i \backslash \{1\} \}$.

\medskip

{\noindent \bf Claim 1:} Let $\delta_M > 0$ be the diameter of $M$. For any $ i \geq 1$, we have 
$$\sys_1(M_i) \geq C_1 \sys(\Gamma_i,X) - C_2 - 2\delta_M.$$

To prove the claim, consider the Dirichlet fundamental domain $D(o)$ of $\Gamma$ in $\Hy^3$ centered in $o$. It is easy to see that  any point $x \in D(o)$ satisfies $d(x,o) \leq \delta_M$. Now let $\alpha_i \subset M_i$ be a closed geodesic realizing the systole of $M_i$. As $M_i \rightarrow M$ is a local isometry, the image of $\alpha_i$ in $M$ has the same length (counted with multiplicity). Denote the image by $\alpha_i$ again.
We can suppose that $x_i \in D(o)$ is a lift of $\alpha_i(0)$. Thus, there exists a unique nontrivial $\gamma_i \in \Gamma_i$ such that $\sys_1(M_i)=d(x_i, \gamma_i(x_i))$.
Note that $d(x_i,o)=d(\gamma_i(x_i),\gamma_i(o)) \leq \delta_M$, therefore, by the triangle inequality we have 
$$ \sys_1(M_i) \geq d(o, \gamma_i(o)) - 2\delta_M \geq C_1 d_X(1,\gamma_i) - C_2 -2\delta_M \geq C_1 \sys(\Gamma_i,X) - C_2 -2\delta_M.$$

Now our problem is reduced to proving that $\sys(\Gamma_i,X)$ grows logarithmically as a function of $[\Gamma:\Gamma_i]$. In order to do so  we use arithmetic of the field $E$ in an essential way.

Let $S(E)$ be the set of all places of $E$, $S_\infty$ be the set of archimedean places, and $S_p$ be the set of places corresponding to the prime ideals $\mathfrak{p}_1, \ldots, \mathfrak{p}_m$, which appear in the definition of $M_i$.
For any $x \in E$, we define the \emph{height} of $x$ by $H(x)= \prod_{v \in S(E)} \max\{1,|x|_v\}$. Recall that for any $x,y \in E$ and an archimedean place $v$, we have $|x+y|_v \leq 4 \max\{|x|_v,|y|_v\}$, and for any non-archimedean place $u$, we have $|x+y|_u \leq \max\{|x|_v,|y|_u\}$. Therefore, the height function satisfies $H(x+y) \leq 4^{\# S_\infty}H(x)H(y)$.

We can generalize the definition of height for matrices with entries in $E$. Thus, for any $M=(m_{ij}) \in \SL_2(E)$, we define $H(M)= \prod_{v \in S(E)} \max\{1,|m_{ij}|_v\}$. We note that $H(M) \geq \max\{H(m_{ij})\}$.

\medskip
  
{\noindent\bf Claim 2:} For any $M,N \in \SL_2(E)$, we have $H(MN) \leq 4^{\# S_\infty} H(M)H(N)$.

\medskip

Indeed, any entry $x$ of $MN$ can be written as $x=au+bt$ with $a,b$ entries of $M$ and $u,t$ entries of $N$. Therefore, for any $v \in S_{\infty}$, $$\max \{1,|x|_v\} \leq 4 \max\{1,|a|_v,|b|_v\} \max \{1,|u|_v,|t|_v\} \leq  4 \max\{1,|m_{ij}|_v\} \max \{1,|n_{ij}|_v\}.$$
For non-archimedean places we have the same inequality without the factor $4$. Now if $MN=(x_{ij})$, then these inequalities show that
$$ H(MN)=\prod_{v \in S(E)} \max\{1,|x_{ij}|_v\} \leq  4^{\# S_\infty} H(M)H(N).$$

Next we want to estimate from below the height of $\gamma$ for any nontrivial $\gamma \in \Gamma_i$. 

\medskip
  
{\noindent\bf Claim 3:} There exists a constant $C_3>0$ such that for any $\gamma \in \Gamma_i \backslash \{1\}$ we have $H(\gamma) \geq C_3 p^{ni}$, where $n=[E:\Q]$.

\medskip

Indeed, let $\gamma= \gamma_{r_1} \cdots \gamma_{r_w(\gamma)} \in \Gamma_i$ be a nontrivial element with $\gamma_{r_j} \in X$ and $w(\gamma)=d_X(1,\gamma)$. We now recall the definition of the group $\Gamma_i$. If we write $\gamma=\left( \begin{matrix} a & b \\ c & d \end{matrix} \right)$, then for any $l=1, \ldots, m$ we have
$$\left( \begin{matrix} \psi_l(a) & \psi_l(b) \\ \psi_l(c) & \psi_l(d) \end{matrix} \right) \equiv \left( \begin{matrix} I_{d_l} & 0 \\ 0 & I_{d_l} \end{matrix}  \right) \mod (p^i \Z_p).$$
By the definition of $\psi_l$ we have that $(a-1)b_j^l, bb_j^l, cb_j^l, (d-1)b_j^l \in p^i \hat{\cO}_{E,\mathfrak{p}_l}$ for any $1\leq j \leq d_l$. 
Taking $C^*= \min_{l,j} \{ |b_j^l|_{\mathfrak{p}_l} \} > 0$, we obtain 
$$C^*\max\{|a-1|_{\mathfrak{p}_l},|b|_{\mathfrak{p}_l},|c|_{\mathfrak{p}_l},|d-1|_{\mathfrak{p}_l} \} \leq \Nr({\mathfrak{p}_l})^{-ie_l},$$ 
for any $l=1, \ldots, m$. 
This is because $|p|_{\mathfrak{p}_l}=\Nr({\mathfrak{p}_l})^{-e_l}$ by definition, where for an ideal $I \subset \cO_E$ the \emph{norm} of $I$ is equal to $\Nr(I)=\# (\cO_E / I)$.

Recall that the Product Formula says that for any nonzero $x \in E$ we have $\prod_{v} |x|_v = 1$. Since $\gamma$ is nontrivial, at least one of the numbers $\{ a-1, b, c, d-1\}$ is not zero.
Therefore, if we apply the Product Formula for any nonzero element in this set, we obtain 
$$ \max \{H(a-1),H(b),H(c),H(d-1)\} \geq \prod_{l=1}^m C^* \Nr({\mathfrak{p}_l})^{ie_i}=(C^*)^m p^{ni}.$$  

Moreover, by the estimate of the height of a sum we have 
$$\max\{H(a-1),H(b),H(c),H(d-1)\} \leq 4^{\# S_\infty} \max\{H(a),H(b),H(c),H(d)\},$$ 
therefore, 
$$ H(\gamma) \geq \max\{H(a),H(b),H(c),H(d)\} \geq \frac{(C^*)^m p^{ni}}{4^{\# S_\infty}}=C_3p^{ni}. $$
This proves Claim~3. 

\medskip

We can now finish the proof of the lemma. If we take $C_4=4^{\# S_\infty} \max\{H(M) \, | \, M \in X \}$, we have
$$ C_3p^{ni} \leq H(\gamma) \leq (4^{\# S_\infty})^{d_X(1,\gamma)-1}(\max\{H(M) \, | \, M \in X \})^{d_X(1,\gamma)} \leq C_4^{d_X(1,\gamma)}.$$
This estimate holds for any nontrivial $\gamma \in \Gamma_i$, hence $ C_3p^{ni} \leq C_4^{\sys(\Gamma_i,X)}$ for any $i$. On the other hand, there exists a constant $C_5>0$ such that $[\Gamma:\Gamma_i] \leq C_5 p^{i\dim(G)}.$
These inequalities together imply that 
$$ \sys(\Gamma_i,X) \geq \frac{n}{\dim(G)\log(C_4)}\log([\Gamma:\Gamma_i]) + \frac{\log(C_3C_5^{\frac{-n}{\dim G}})}{\log(C_4)}.$$
Since $[\Gamma:\Gamma_i] \to \infty$ and $\sys_1(M_i)$ is bounded below by a positive constant, we conclude that 
there exists a constant $c_1=c_1(o,\delta_M,p,\psi_1, \ldots, \psi_m)=c_1(M) > 0$ such that $\sys_1(M_i) \geq c_1 \log([\Gamma:\Gamma_i])$ for any $i \geq 1$.
\end{proof}

Note that the constant $c_1$ depends on $M$ (cf. Question~\ref{quest}). If $M$ is arithmetic, then by \cite{KSV} we can take $c_1 = \frac23 - \epsilon$ for a small $\epsilon > 0$ assuming $n_i$ is sufficiently large. In general case the argument of \cite{KSV} does not apply, while the proof of Lemma~\ref{lem} does not provide a sufficient level of control over the constants. 

\section{Proofs of Theorems \ref{thm1} and \ref{thm2}}\label{sec3}

Following \cite{Bel}, we define the \emph{systolic genus} of a manifold $M$ by 
$$\sysg(M) = \min\{g \mid \text{the fundamental group } \pi_1(M) \text{ contains } \pi_1(S_g)\},$$
where $S_g$ denotes a closed Riemann surface of genus $g > 0$.

Let $M$ be a closed hyperbolic $3$-manifold with sufficiently large systole $\sys_1(M)$.  By \cite[Theorem~2.1]{Bel}, we have
\begin{align}\label{eq3:1}
\log \sysg(M) \ge c_2\cdot\sys_1(M),
\end{align}
where $c_2 > 0 $ is an absolute constant (for any $\delta>0$, assuming $\sys_1(M)$ is sufficiently large, we can take $c_2 = \frac12-\delta$). 

The second ingredient of the proof is a theorem of Calegary--Emerton \cite{CE}, which implies that for the sequences of covers defined in Section~\ref{sec:prelim} we have  
\begin{align}\label{eq3:2}
\dim \mathrm{H}_1(M_i, \F_p) \ge \lambda \cdot p^{(d-1)i} + O(p^{(d-2)i})
\end{align}
for some rational constant $\lambda \neq 0$. Recall that we have dimension $d = \dim(G) \ge 6$ and the degree of the covers $M_i \to M$ grows like $p^{di}$. Hence we can rewrite \eqref{eq3:2} in the form 
\begin{align}\label{eq3:3}
\dim \mathrm{H}_1(M_i, \F_p) \ge c_3\vol(M_i)^{5/6},
\end{align}
where $c_3 > 0$ is a constant depending on $M$ and we assume that $\vol(M_i)$ is sufficiently large.  

We note that in contrast with the previous related work, the theorem of \cite{CE} applies to non-arithmetic manifolds as well as to the arithmetic ones. 

Now recall a result of Baumslag--Shalen \cite[Appendix]{BSh}. They show that if $\sysg(M) \ge k$ and $\dim \mathrm{H}_1(M, \Q) \ge k + 1$, then $\pi_1(M)$ is $k$-free. In a subsequent paper \cite{SW}, Shalen and Wagreich proved that the same conclusion holds if $\sysg(M) \ge k$ and $\dim \mathrm{H}_1(M, \F_p) \ge k + 2$ [loc. cit., Proposition~1.8].

We now bring all the ingredients together. Given a closed hyperbolic $3$-orbifold $M$, for the sequence $(M_i)$ of its manifold covers defined in Section~\ref{sec:prelim} we have:
\begin{align*}
\sysg(M_i) & \ge e^{c_2\cdot \sys_1(M_i)} \text{ (by \eqref{eq3:1}) }\\
                & \ge \vol(M_i)^c \text{ (by Lemma~\ref{lem}); }
\end{align*}
and
\begin{align*}
\dim \mathrm{H}_1(M_i, \F_p) & \ge c_3\cdot\vol(M_i)^{5/6} \text{ (by \eqref{eq3:3})}.
\end{align*}
Hence by the theorem from \cite{SW} cited above we obtain
\begin{align*}
\Nfr(\pi_1(M_i)) \ge \vol(M_i)^C,
\end{align*}
where $C = C(M) > 0$ and we assume that $\vol(M_i)$ is sufficiently large. This proves Theorem~\ref{thm1}.

For the second theorem recall that the systole of a hyperbolic $3$-manifold is bounded above by the logarithm of its volume. Indeed, a manifold $M$ with a systole $\sys_1(M)$ contains a ball of radius $r = \sys_1(M)/2$. The volume of a ball in $\Hy^3$ is given by $\vol(B(r)) = \pi(\sinh(2r)-2r)$, hence we get
\begin{align*}
& \vol(M) \ge \pi (\sinh(\sys_1(M)) - \sys_1(M)) \sim \frac{\pi}2 e^{\sys_1(M)};\\
& \vol(M) \ge e^{c\cdot \sys_1(M)}, \text{ as } \sys_1(M)\to\infty.
\end{align*}
By Lemma~\ref{lem}, the systole of the covers $M_i\to M$ grows as $i \to\infty$. Therefore, we can bound both $\sysg(M_i)$ and  $\dim \mathrm{H}_1(M_i, \F_p)$ below by an exponential function of $\sys_1(M_i)$ with an absolute constant in exponent. Theorem~\ref{thm2} now follows immediately from the theorem of \cite{SW}. 
\qed

\begin{rmk}
It follows from the proof that for any $\delta >0$, assuming $\sys_1(M_i)$ is large enough, we can take $\varepsilon$ in Theorem~\ref{thm2} equal to $e^{\frac12-\delta} - 1$. The same bound applies for the constant in Theorem~\ref{thm3}, which we prove in the next section.
\end{rmk}

\section{Generalization to finite volume hyperbolic \texorpdfstring{$3$}{3}-manifolds}\label{sec4}

Let $\Gamma < \PSL_2(\C)$ be a finite covolume Kleinian group. The quotient $M = \Hy^3/\Gamma$ is a finite volume orientable hyperbolic $3$-orbifold, which can be either closed or non-compact with a finite number of cusps. The group $\Gamma$ is a relatively hyperbolic group with respect to the cusp subgroups. In this section we discuss a generalization of Gromov's conjecture and our results to this class of groups. 

We call $\Gamma$ a \emph{$k$-semifree  group} if any subgroup of $\Gamma$ generated by $k$ elements is a free product of free abelian groups. The maximal $k$ for which $\Gamma$ is $k$--semifree is denoted by $\Nfr'(\Gamma)$. With this definition, we can generalize Gromov's conjecture to relatively hyperbolic groups. Although the injectivity radius of manifolds with cusps vanish, their systole is still bounded away from zero. Therefore, a natural generalization of Gromov's conjecture would be that $\Nfr'(\Gamma)$ is bounded
below by an exponential function of the systole of the associated quotient space $M$. Theorem~\ref{thm3}, which we prove in this section, can be considered as an evidence for this conjecture. 

We need to modify the definition of the \emph{systolic genus} of a manifold $M$ in the following way:  
$$\sysg(M) = \min\{g > 1 \mid \text{the fundamental group } \pi_1(M) \text{ contains } \pi_1(S_g)\},$$
where $S_g$ denotes a closed Riemann surface of genus $g$. We excluded the genus $g = 1$ in order to adapt the definition to the non-compact finite volume $3$-manifolds which otherwise would all have $\sysg = 1$.

Let $M$ be a finite volume hyperbolic $3$-manifold with sufficiently large systole $\sys_1(M)$.  By \cite[Theorem~2.1]{Bel}, if $M$ is closed, we have
\begin{align}\label{eq4:1}
\log \sysg(M) \ge c_2\cdot\sys_1(M),
\end{align}
where $c_2 > 0 $ is an absolute constant. We now discuss a generalization of this result to non-compact finite volume $3$-manifolds. The first step in the proof of the theorem in \cite{Bel} is an application of the theorem of Schoen--Yau and Sacks--Uhlenbeck, which allows to homotop a $\pi_1$-injective map of a surface of genus $g > 1$ into $M$ to a minimal immersion. 
This result was recently generalized to the finite volume hyperbolic $3$-manifolds in the work of Collin--Hauswirth--Mazet--Rosenberg \cite{CHMR17} and Huang--Wang \cite{HW17} (see in particular \cite[Theorem~1.1]{HW17}). So let $S_g$ be a closed immersed least area minimal surface in $M$. In order to establish \eqref{eq4:1} for $M$ we can suppose that $S_g$ is embedded. Indeed, since $\pi_1(M)$ is \emph{LERF} \cite[Corollary 9.4]{Agol12} there exists a finite covering $\tilde{M}$ of $M$ such that $S_g$ is embedded and \pinj in $\tilde{M}$. Moreover, $g \geq \sysg(\tilde{M})$ and $\sys_1(\tilde{M}) \geq \sys_1(M)$. If $S_g$ has no accidental parabolic curves, then the systole of $S_g$ with respect to the induced metric satisfies $\sys_1(S_g) \ge \sys_1(M)$ and the rest of the proof in \cite{Bel} applies without any changes.

In the presence of accidental parabolics, we can apply the following lemma.

\begin{lemma}[Compression Lemma] \label{compressionlemma}
Let $M$ be a non-compact hyperbolic $3$-manifold of finite volume. Suppose that there exists a $\pi_1$-injective embedded closed surface $S_g \subset M$, for some genus $g \geq 2$, such that $S_g$ has an accidental parabolic simple curve $\alpha$. Then there exist disjoint tori $T_1, \ldots, T_n \subset M$, one for each cusp $\mathcal{C}(T_i)$ of $M$, such that the compact $3$-manifold $M'=M \setminus \displaystyle \cup_{i=1}^n \mathcal{C}(T_i)$  has a properly incompressible and boundary-incompressible surface $S_{g',p}$ with $g' \geq \frac{g}{2}$ and $1 \leq p \leq 2$.
 \end{lemma}
\begin{proof}
Suppose that $\alpha$ is associated to a parabolic isometry corresponding to a cusp $\mathcal{C} = T_0 \times [0, \infty)$ of $M$, where $T_0$ is a maximal torus. 
Since $S_g$ is compact we can consider a torus $T = T_0 \times \{t_0\} \subset \mathcal{C}$ for some $t_0 >0$ sufficiently large such that $S_g \subset M \setminus T_0 \times [t_0, \infty)$. We denote by $\beta \subset T$ the corresponding simple curve homotopic to $\alpha$.

We first show that there exists an embedding $f:S_g \rightarrow M$ homotopic to the embedding $\iota:S_g \rightarrow M$ such that $f$ is transversal to some torus $T_1 \subset \mathcal{C}$ and $f(S_g) \cap T_1 \times  [0, \infty) \subset \mathcal{C}$ is an annulus with boundary curves $f(\alpha_0), \, f(\alpha_1)$, where $\alpha_0, \alpha_1$ are the boundary curves of a collar neighborhood of $\alpha$ in $S_g$.   
  
As an application of the Jaco--Shalen Annulus Theorem \cite[Theorem VIII.13]{Jaco80}, there exists an embedding $H_0: \sss \times [0,1] \rightarrow M $ with $H(\theta,0)=\alpha(\theta)$ and $H(\theta,1)=\beta(\theta)$ (see \cite[Lemma 2.1]{OT03}).
We can suppose that $H_0$ is transversal to $S_g$ and $T$ and is such that if we denote by $\mathcal{A}$ the image $H_0(\sss \times [0,1])$, then $\mathcal{A} \cap S_g =\alpha$ and $\mathcal{A} \cap M \setminus T \times [0, \infty) = \beta$. 

Let $\mathcal{D}$ be a collar neighborhood of $\alpha$ in $S_g$ contained in  a tubular neighborhood $\pi: E \subset M \rightarrow \mathcal{A}$ such that $\mathcal{D} \cap \mathcal{A} = \alpha$. Since $\pi:E \rightarrow \mathcal{A}$ is trivial, we can deform $\mathcal{D}$ into $E$ preserving the boundary and moving $\alpha$ along $\mathcal{A}$. We get a new annulus $\mathcal{D}' \subset M$ with $\partial \mathcal{D}' = \alpha_0 \cup \alpha_1$ and $\mathcal{D}' \cap T = \beta$. 

Let $\psi$ be the diffeomorphism between $\mathcal{D}$ and $\mathcal{D}'$ given by the deformation. We can suppose that $\psi$ is the identity in a small neighborhood of the boundary. We now define the map $f:S_g \rightarrow M$ by $f(x)=x$ if $x \notin \mathcal{D}$ and $f(y)=\psi(y)$ if $ y \in \mathcal{D}$. It is a smooth embedding homotopic to the inclusion.

By transversality, for some $ 0<t_1<t_0$ we have a torus $T_1 = T_0 \times \{t_1\}$ and a subannulus $\hat{\mathcal{D}} \subset \mathcal{D}$ such that  $f$ is transversal to $T_1$ and 
$$f(S_g) \cap M \setminus T_1 \times [0, \infty) = f(S_g \setminus \mathrm{int}(\hat{\mathcal{D}})) \hspace{0.3cm} \mbox{and} \hspace{0.3cm}f(\partial (S_g \setminus \mathrm{int}(\hat{\mathcal{D}}))) = f(\partial \hat{\mathcal{D}}) \subset T_1.$$ 
This shows that embedding $f$ has the desired properties.  

Now, for the torus $T_1$ constructed above,  there exist disjoint tori $T_2, \ldots, T_n$ in the cusps of $M$ such that the corresponding cusps $\mathcal{C}(T_j) \cap \mathcal{C}(T_1) = \emptyset $ for all $j=2, \ldots, n$ and $f(S_g \setminus \mathrm{int}(\hat{\mathcal{D}})) \subset M'=M \setminus \displaystyle \cup_{i=1}^n \mathcal{C}(T_i)$, and we  have that  $f(S_g \setminus \mathrm{int}(\hat{\mathcal{D}})) \subset M'$ is a proper submanifold of $M'$. 

Note that $f(S_g \setminus \mathrm{int}(\hat{\mathcal{D}}))$ is connected with two boundary curves if $\alpha$ does not separate and has two components with a boundary curve if $\alpha$ separates it.                                                                                                                                                                                                                                                                                                                                                                                                                                                                                                                                                                                                                                                                                                                                                                                          In the latter case we consider the component with the maximal genus. Hence in both cases we have a surface $S_{g',p}$ with $g' \geq \frac{g}{2}$ and $1 \leq p \leq 2$ and a proper embedding $f:S_{g',p} \rightarrow M'$.

Recall that a properly embedded surface $F$ in a compact $3$-manifold $N$ with boundary is called \emph{boundary-compressible} if either $F$ is a disk and $F$ is parallel to a disk in $\partial N$, or $F$ is not a disk and there exists a disk $D \subset N$ such that $D \cap F=c $ is an arc in $\partial D$, $D \cap \partial N=c'$ is an arc in $\partial D$, with $c \cap c'= \partial c = \partial c'$ and $c \cup c' = \partial D$, and either $c$ does not separate $F$ or $c$ separates $F$ into two components and the closure of neither is a disk. Otherwise, $F$ is \emph{boundary-incompressible} (see \cite[Chapter~III]{Jaco80}).

Since $S_g \subset M$ is $\pi_1$-injective, it follows from the definition and our construction that $S_{g',p} \subset M'$ is incompressible and boundary-incompressible.
\end{proof}

We now apply to $S_{g', p}$ a result of Adams and Reid \cite[Theorem~5.2]{AR00}. Since $\sys_1(M)=\sys_1(M')$, it immediately implies inequality \eqref{eq4:1}.

The theorem of Calegary--Emerton applies to non-cocompact groups as well as to the cocompact ones.

We finally recall a result of Anderson--Canary--Culler--Shalen \cite{ACCS96}. They show that if $\sysg(M) \ge k$ and $\dim \mathrm{H}_1(M, \F_p) \ge k + 2$ for some prime $p$, then $\pi_1(M)$ is $k$-semifree [loc. cit., Corollary~7.4]. This theorem generalizes the previous results in \cite{BSh, SW} to non-compact hyperbolic $3$--manifolds. Its  proof also makes an essential use of topology of $3$-manifolds. 

Similar to Section~\ref{sec3}, we bring together all the ingredients considered above. 

Given a finite volume  hyperbolic $3$-orbifold $M$, for the sequence $(M_i)$ of its manifold covers defined in Section~\ref{sec:prelim} we have:
\begin{align*}
\sysg(M_i) & \ge e^{c_2\cdot \sys_1(M_i)} \text{ (by \eqref{eq4:1}),}
\end{align*}
and
\begin{align*}
\dim \mathrm{H}_1(M_i, \F_p) & \ge c_3\cdot\vol(M_i)^{5/6} \text{ (by Calegary--Emerton).}
\end{align*}

The fact that a manifold $M$ with systole $\sys_1(M)$ contains a ball of radius $r = \sys_1(M)/2$ is not necessarily true for non-compact finite volume hyperbolic $3$-manifolds but it is still possible to bound the volume by an exponential function of the systole. By Lakeland--Leininger \cite[Theorem~1.3]{LL14}, we have
\begin{align*}
\vol(M) \ge e^{c\cdot \sys_1(M)}, \text{ as } \sys_1(M)\to\infty
\end{align*}
(with $c = \frac{3}{4}-\delta$ for any $\delta > 0$, assuming $\sys_1(M)$ is sufficiently large). 

Although we do not have a generalization of Lemma~\ref{lem}, we do know that $\sys_1(M_i)\to\infty$ with $i$ because the sequence of covers $M_i \to M$ is co-final. Therefore, we can bound both $\sysg(M_i)$ and  $\dim \mathrm{H}_1(M_i, \F_p)$ below by an exponential function of $\sys_1(M_i)$ with an absolute constant in exponent and Theorem~\ref{thm3} now follows from the theorem of \cite{ACCS96}. \qed

\medskip

{\noindent\bf Acknowledgments.} We would like to thank Valdir Pereira Jr. for helpful discussions. We thank Alan Reid for his critical remarks on the first version of this paper.

\end{document}